\documentclass[11pt,CJK]{article}
\usepackage[numbers,sort&compress]{natbib}
\usepackage{enumerate}
\usepackage{amscd}
\usepackage{array}
\usepackage{amsmath}
\usepackage{latexsym}
\usepackage{amsfonts}
\usepackage{amssymb}
\usepackage{amsthm}
\usepackage{verbatim}
\usepackage{mathrsfs}
\usepackage{enumerate}
\usepackage{hyperref}

\theoremstyle{plain}
\theoremstyle{definition}\newtheorem{theorem}{Theorem}[section]
\theoremstyle{plain}\newtheorem{lemma}[theorem]{Lemma}
\theoremstyle{plain}
\theoremstyle{plain}
\theoremstyle{remark}\newtheorem{remark}{Remark}[section]
\theoremstyle{definition}
\theoremstyle{plain}
\usepackage{xcolor}
\newcommand{\wred}[1]{\textcolor{black}{#1}}

\newcommand{\Div}{\mathrm{div}\,}
\newcommand{\B}{\Big}

\newcommand{\be}{\begin{equation}}
\newcommand{\ee}{\end{equation}}
 \newcommand{\ba}{\begin{aligned}}
 \newcommand{\ea}{\end{aligned}}

\providecommand{\bysame}{\leavevmode\hbox to3em{\hrulefill}\thinspace}
  \newcommand{\f}{\frac}
    
  \newcommand{\ben}{\begin{enumerate}}
   \newcommand{\een}{\end{enumerate}}

\newcommand{\ti}{\nabla}

\newcommand{\Rmnum}[1]{\expandafter\@slowromancap\romannumeral #1@}

\allowdisplaybreaks

\numberwithin{equation}{section}
\begin{document}
\title{
 New regularity criteria based on pressure or gradient of velocity  in Lorentz spaces   for the 3D Navier-Stokes equations }
\author{ Xiang Ji\footnote{ Department of Mathematics and Information Science, Zhengzhou University of Light Industry, Zhengzhou, Henan  450002,  P. R. China Email: xiangji201409@163.com},\;~Yanqing Wang\footnote{ Department of Mathematics and Information Science, Zhengzhou University of Light Industry, Zhengzhou, Henan  450002,  P. R. China Email: wangyanqing20056@gmail.com}  ~  and\, Wei Wei\footnote{Center for Nonlinear Studies, School of Mathematics, Northwest University, Xi'an, Shaanxi 710127, P. R. China Email: ww5998198@126.com }
 }
\date{}
\maketitle
\begin{abstract}
  In  this paper,   we derive regular criteria  via pressure or gradient of the velocity  in Lorentz spaces  to the 3D Navier-Stokes equations. It is  shown  that a  Leray-Hopf weak solution  is regular on $(0,T]$ provided that either the norm $\|\Pi\|_{L^{p,\infty}(0,T; L ^{q,\infty}(\mathbb{R}^{3}))} $ with $  {2}/{p}+{3}/{q}=2$ $( {3}/{2}<q<\infty)$ or $\|\nabla\Pi\|_{L^{p,\infty}(0,T; L ^{q,\infty}(\mathbb{R}^{3}))} $ with $  {2}/{p}+{3}/{q}=3$ $(1<q<\infty)$ is small. This gives  an affirmative  answer to a question proposed by Suzuki in
 \cite[Remark 2.4, p.3850]{[Suzuki2]}.  Moreover, regular conditions  in terms of $\nabla u$   obtained here generalize known ones to    allow the time direction to belong  to Lorentz
spaces.
 \end{abstract}
\section{Introduction}
\label{intro}
\setcounter{section}{1}\setcounter{equation}{0}
We focus our attention on the 3D  Navier-Stokes system
\be\left\{\ba\label{NS}
&u_{t} -\Delta  u+ u\cdot\ti
u  +\nabla \Pi=0, \\
&\Div u=0,\\
&u|_{t=0}=u_0,
\ea\right.\ee
 where \wred{the unknown vector $u=u(x,t)$ describes the flow  velocity field}, the scalar function $\Pi$ represents the   pressure.
 The  initial datum $u_{0}$ is given and satisfies the divergence-free condition.

 In pioneering works \cite{[Leray1],[Hopf]}, Leray  and Hopf proved that,
 for any given  divergence-free data
$u_0\in L^2(\Omega)$,
  there exists a global  weak solution $ u $ of the 3D Navier-Stokes equations such that
  $u\in L^{\infty}(0, \infty; L^{2}(\Omega))\cap L^{2}(0, \infty; W^{1,2}(\Omega)).$ However, the full regularity of Leray-Hopf weak solutions  to  \eqref{NS} is unknown. Partial regularity such as regularity criteria of Leray-Hopf weak solutions was extensively
  studied (see \cite{[Beirao da Veiga],[Beirao da Veiga2],[BV],[BPR],[CFZ],[BG],[CF],[CP],[CWY],[ESS],[HW],[HW2],[KK],[Kozono],[Phuc],
  [Serrin],[Struwe],[Struwe2],[Suzuki1],[Suzuki2],[Sohr],[Takahashi],[Yuan],[Zhou1],[Zhou2],[Zhou3]}
    and references therein). In particular, a weak solution $u$ is smooth on $(0,T]$ if $u$  satisfy  one of the following four  conditions
  \begin{enumerate}[(1)]
 \item Serrin \cite{[Serrin]}, Struwe \cite{[Struwe]}, Escauriaza,  Seregin and  \v{S}ver\'{a}k \cite{[ESS]}
  \be \label{serrin1}
  u\in  L^{p} (0,T;L^{q}( \mathbb{R}^{3})) ~~~ \text{with}~~~~2/p+3/q=1, ~~q\geq3.
  \ee
 \item  Beirao da Veiga \cite{[Beirao da Veiga]}
  \be \label{serrin2}
  \nabla u\in  L^{p} (0,T;L^{q}( \mathbb{R}^{3})) ~~~ \text{with}~~~~2/p+3/q=2 , ~~q>3/2.
  \ee
 \item  Berselli and  Galdi \cite{[BG]},   Struwe \cite{[Struwe2]}, Zhou \cite{[Zhou1],[Zhou2]}
  \be \label{serrin3}
 \Pi \in  L^{p} (0,T;L^{q}( \mathbb{R}^{3})) ~~~ \text{with}~~~~2/p+3/q=2 , ~~q>3/2.
  \ee
\item Berselli and  Galdi \cite{[BG]},   Struwe \cite{[Struwe2]}, Zhou \cite{[Zhou1],[Zhou2],[Zhou3]}
  \be \label{serrin4}
 \nabla\Pi \in  L^{p} (0,T;L^{q}( \mathbb{R}^{3})) ~~~ \text{with}~~~~2/p+3/q=3 , ~~q>1.
  \ee
  \end{enumerate}
 On the other hand, the Navier-Stokes system  enjoys the scale invariance. Indeed, if the pair $(u(x,t),\Pi(x,t))$    solves  system \eqref{NS}, then the pair $( u_{\lambda},\Pi_{\lambda})  $ is also a solution of \eqref{NS} for any $\lambda\in \mathbb{R}^{+},$ where
\be\label{scaling}
u_{\lambda}=\lambda u(\lambda x,\lambda^{2}t),~~~~~\Pi_{\lambda}=\lambda^{2} \Pi(\lambda x,\lambda^{2}t).
\ee
  We would like to mention that the  norm  $\|\cdot\|_{L^{p}(0,\infty;L^{q}( \mathbb{R}^{3}))}$ with $2/p+3/q=1$  is scaling invariant for $u$ under the natural scaling \eqref{scaling}. Similarly, for gradient $\nabla u$ or pressure $\Pi$, the norm  $\|\cdot\|_{L^{p}(0,\infty;L^{q}( \mathbb{R}^{3}))}$ with $2/p+3/q=2$  is also  scaling invariant.
  Therefore,  all the norms in \eqref{serrin1}-\eqref{serrin4} have the scale invariance. It is well-known that
Lorentz spaces $L^{r,s}(\mathbb{R}^{3})$ $(s\geq r)$ are larger than the  Lebesgue
spaces $L^{r }(\mathbb{R}^{n})$. Moreover,  similarly  to the  spaces  $L^{r}(\mathbb{R}^{n})$, notice that there holds
$\|f(\lambda \cdot)\|_{L^{r,s}(\mathbb{R}^{n})}=
\lambda ^{-\f{n}{r}}\|f(\cdot)\|_{L^{r,s}(\mathbb{R}^{n})}$.
Therefore, a natural question arises whether results  \eqref{serrin1}-\eqref{serrin4} still hold in Lorentz spaces.
 Indeed, some corresponding regularity criteria involving Lorentz spaces have been established as follows:
  There exists a  positive  constant $\varepsilon$ such that a weak solution  $u$ is smooth on $(0,T]$  if  $u$ satisfy one of the following  four  conditions
  \begin{enumerate}[(1)]
 \item Takahashi   \cite{[Takahashi]}; Chen and
Price \cite{[CP]}  and Sohr \cite{[Sohr]};
 Kozono  and   Kim  \cite{[KK]}; Bosia,   Pata and   Robinson \cite{[BPR]}
  \be \label{serrin1L}
  u\in  L^{p,\infty} (0,T;L^{q,\infty}( \mathbb{R}^{3})) ~\text{and}~  \|u\|_{L^{p,\infty} (0,T;L^{q,\infty}( \mathbb{R}^{3}))}\leq\varepsilon~ \text{with}~2/p+3/q=1, ~~q>3.
  \ee
 \item   He and Wang  \cite{[HW]}
  \be \label{serrin2L}
  \nabla u\in  L^{p} (0,T;L^{q,\infty}( \mathbb{R}^{3}))   ~~\text{with}~~~~2/p+3/q=2 , ~~q>3/2.
  \ee
 \item Suzuki \cite{[Suzuki1],[Suzuki2]}
  \be \label{serrin3L}
 \Pi \in  L^{p,\infty} (0,T;L^{q,\infty}( \mathbb{R}^{3})) ~ \text{and}~  \|\Pi\|_{L^{p,\infty} (0,T;L^{q,\infty}( \mathbb{R}^{3}))}\leq\varepsilon ~\text{with}~2/p+3/q=2, ~5/2<q<\infty.
  \ee
\item Suzuki \cite{[Suzuki2]}
  \be \label{serrin4L}
 \nabla\Pi \in  L^{p,\infty} (0,T;L^{q,\infty}( \mathbb{R}^{3})) ~ \text{and}~  \|\nabla\Pi\|_{L^{p,\infty} (0,T;L^{q,\infty}( \mathbb{R}^{3}))}\leq\varepsilon~ \text{with}~2/p+3/q=3,  5/3\leq q<3.
  \ee
  \end{enumerate}
In
 \cite[Remark 2.4, p.3850]{[Suzuki2]}, Suzuki proposed a question whether   the case $3/2<q<5/2$  guarantees the regularity of the  Leray-Hopf weak solutions. The first objective of  our  paper is to give a positive answer to this issue. Before formulating our results, we mention that,  as \eqref{serrin2L}, regularity criteria in terms of pressure $\Pi$ or gradient of pressure  $\nabla \Pi$  with only  space direction belonging to Lorentz spaces can be found in \cite{[Yuan],[CFZ]}.
As for other regularity criteria involving Lorentz spaces, see \cite{[BV],[WZ],[CWY],[HW],[Phuc],[PY]}.  Now our  first result is stated as follows.
\begin{theorem}\label{the1.1}
Suppose that $(u,\,\Pi)$ is a   weak solution to \eqref{NS} with the   divergence-free  initial data $u_{0}(x)\in L^{2}(\mathbb{R}^{3})\cap L^{4}(\mathbb{R}^{3})$. Then there exists a positive constant $\varepsilon_{1}$ such that $u(x,t)$ is a regular solution on $(0,T]$ provided that one  of the following two conditions holds
\begin{enumerate}[(1)]
 \item \label{coro1}
    $  \Pi \in L^{p,\infty}(0,T; L ^{q,\infty}(\mathbb{R}^{3}))$ ~and~ $$\|\Pi\|_{L^{p,\infty}(0,T; L ^{q,\infty}(\mathbb{R}^{3}))} \leq\varepsilon_{1}, ~ \text{with} ~~2/p+3/q=2 , ~3/2<q<\infty;  $$
 \item $\nabla \Pi \in L^{p,\infty}(0,T; L ^{q,\infty}(\mathbb{R}^{3}))$ and  $$\|\nabla \Pi\|_{L^{p,\infty}(0,T; L ^{q,\infty}(\mathbb{R}^{3}))} \leq\varepsilon_{1}, ~~~~ \text{with} ~~~ 2/p+3/q=3 , ~1<q<\infty. $$
 \end{enumerate}
\end{theorem}
In \cite{[Suzuki1],[Suzuki2]}, Suzuki proved \eqref{serrin3L}-\eqref{serrin4L}  via the truncation method introduced by  Beirao da Veiga  in \cite{[Beirao da Veiga2]}.
Here, taking advantage of generalized
Gronwall lemma due to Bosia,   Pata and   Robinson \cite{[BPR]} and some appropriate interpolation inequalities, we deduce  Theorem \ref{the1.1}.  To this end, we  develop the technique of \cite{[BPR]} in a general way.
It seems that  the  arguments  in Theorem \ref{the1.1} and  in the following  Theorem \ref{the1.2} can be applied  to other incompressible fluid equations  such as
magnetohydrodynamic equations.

Next target of  our  paper is to improve  \eqref{serrin2L} to allow the time direction  to  belong to Lorentz spaces, which is  partially  inspired by   \cite{[WWY]}. Very recently, authors in \cite{[WWY]} established the local
regularity criteria in terms of $\nabla u$ in $L^{p,\infty}(0,T; L ^{q,\infty}(B(1)))$ with sufficiently  small  $\|\nabla u\|_{L^{p,\infty}(0,T; L ^{q,\infty}(B(1)))}$, where the pair $(p,q)$ satisfies $2/p+3/q=2, ~~3/2<q< \infty.$
We will present its whole space case  in the following.
\begin{theorem}\label{the1.2}
Suppose that $(u,\,\Pi)$ is a   weak solution to \eqref{NS} with the   divergence-free  initial data $u_{0}(x)\in L^{2}(\mathbb{R}^{3})\cap W^{1,2}(\mathbb{R}^{3})$. Then there exists a positive constant $\varepsilon_{2}$ such that $u(x,t)$ is a regular solution on $(0,T]$ if
    $  \nabla u \in L^{p,\infty}(0,T; L ^{q,\infty}(\mathbb{R}^{3}))$ and $$\|\nabla u \|_{L^{p,\infty}(0,T; L ^{q,\infty}(\mathbb{R}^{3}))} \leq\varepsilon_{3}, ~~~~ \text{with} ~~~ 2/p+3/q=2 , ~~ 3/2<q< \infty. $$
 \end{theorem}
\begin{remark}
Theorem \ref{the1.2} is a generalization of  \eqref{serrin2}   and \eqref{serrin2L}.
\end{remark}
\begin{remark}
 Utilizing  the boundedness of Riesz Transform  \eqref{brl} in Lorentz spaces,
$\nabla u $  can be replaced  by its  symmetric  part $\f{1}{2}(\nabla u+\nabla u^{^{\text{T}}})$ or its antisymmetric part $\f{1}{2}(\nabla u-\nabla u^{^{\text{T}}}) $.
\end{remark}

\section{Notations and  some auxiliary lemmas} \label{section2}
First, we introduce some notations used in this paper.
 For $p\in [1,\,\infty]$, the notation $L^{p}(0,T;X)$ stands for the set of measurable functions $f(x,t)$ on the interval $(0,T)$  with values in $X$ and $\|f(\cdot,t)\|_{X}$ belonging to $L^{p}(0,T)$.
 The classical Sobolev space $W^{k,2}(\mathbb{R}^{3})$ is equipped with the norm $\|f\|_{W^{k,2}(\mathbb{R}^{3})}=\sum\limits_{\alpha =0}^{k}\|D^{\alpha}f\|_{L^{2}(\mathbb{R}^{3})}$. $|E|$ represents the $n$-dimensional Lebesgue measure of a set $E\subset \mathbb{R}^{n}$. We will use the summation convention on repeated indices.
 $C$ is an absolute constant which may be different from line to line unless otherwise stated in this paper.

Next, we present some basic facts on Lorentz spaces.
For $p,q\in[1,\infty]$, we define
$$
\|f\|_{L^{p,q}(\Omega)}=\left\{\ba
&\B(p\int_{0}^{\infty}\alpha^{q}|\{x\in \Omega:|f(x)|>\alpha\}|^{\f{q}{p}}\f{d\alpha}{\alpha}\B)^{\f{1}{q}} , ~~~q<\infty, \\
 &\sup_{\alpha>0}\alpha|\{x\in \Omega:|f(x)|>\alpha\}|^{\f{1}{p}} ,~~~q=\infty.
\ea\right.
$$	
Furthermore,
$$
L^{p,q}(\Omega)=\big\{f: f~ \text{is  a measurable function  on}~ \Omega ~\text{and} ~\|f\|_{L^{p,q}(\Omega)}<\infty\big\}.
$$ 		
Similarly, one can  define
Lorentz spaces $L^{p,q}(0,T;X)$ in time for $ p\leq q \leq\infty$. $f\in L^{p,   q}(0,T;X)$ means that $\|f\|_{L^{p,q}(0,T;X)}<\infty$, where
$$\|f\|_{L^{p,q}(0,T;X)}=\left\{\ba
&\B( p \int_{0}^{\infty}\alpha^q|\{t\in(0,T)
:\|f(t)\|_{X}>\alpha\}|^{\f{q}{p}}\f{d\alpha}{\alpha}\B)^{\f{1}{q}} , ~~~q<\infty, \\
 &\sup_{\alpha>0}\alpha|\{t\in(0,T)
:\|f(t)\|_{X}>\alpha\}|^{\f{1}{p}} ,~~~q=\infty.\ea\right.
$$
We list the properties of Lorentz spaces  as follows.
\begin{itemize}
\item Interpolation characteristic of Lorentz spaces \cite{[BL]}
\be\label{Interpolation characteristic}
(L^{p_{0},q_{0}}(\mathbb{R}^{n}),L^{p_{1},q_{1}}(\mathbb{R}^{n}))_{\theta,q}=L^{p,q}(\mathbb{R}^{n})
~~~~\text{with}~~~ \f{1}{p}=\f{1-\theta}{p_{0}}+\f{\theta}{p_{1}},~0<\theta<1.\ee

\item
Boundedness of Riesz Transform in Lorentz spaces \cite{[CF]}
\be\|R_{j}f\|_{L^{p,q}(\mathbb{R}^{n})}\leq C\| f\|_{L^{p,q}(\mathbb{R}^{n})},~1<p<\infty.\label{brl}\ee

\item
H\"older's inequality in Lorentz spaces  \cite{[Neil]}
\be\label{hiL}\ba
 &\|fg\|_{L^{r,s}(\mathbb{R}^{n})}\leq \|f\|_{L^{r_{1},s_{1}}(\mathbb{R}^{n})}\|g\|_{L^{r_{2},s_{2}}(\mathbb{R}^{n})},
\\
&\f{1}{r}=\f{1}{r_{1}}+\f{1}{r_{2}},~~\f{1}{s}=\f{1}{s_{1}}+\f{1}{s_{2}}.
\ea\ee
\item

The Lorentz spaces increase as the exponent $q$ increases \cite{[Grafakos],[Maly]}

For $1\leq p\leq\infty$ and $1\leq q_{1}<q_{2}\leq\infty,$
\be\label{Lorentzincrease}
\|f\|_{L^{p,q_{2}}(\mathbb{R}^{n})}\leq \B(\f{q_{1}}{p}\B)^{\f{1}{q_{1}}-\f{1}{q_{2}}}\|f\|_{L^{p,q_{1}}(\mathbb{R}^{n})}.
\ee

\item Sobolev inequality in Lorentz spaces \cite{[Neil],[Tartar]}
\be\label{sl}
\|f\|_{L^{\f{np}{n-p},p}(\mathbb{R}^{n})}\leq \|\nabla f\|_{L^{p}(\mathbb{R}^{n})}~~\text{with}~~ 1\leq p<n.\ee
\end{itemize}
 Additionally, we recall the following useful
Gronwall lemma first shown by  Bosia,   Pata and   Robinson in \cite{[BPR]}.
  \begin{lemma}[\cite{[BPR]}]\label{2.1}
 Let $\phi$ be a measurable positive function defined on the interval $[0,T]$. Suppose that there  exists  $\kappa_{0}>0$ such that for all $0<\kappa<\kappa_{0}$ and a.e. $t\in[0,T]$, $\phi$ satisfies the inequality
 $$\f{d}{dt}\phi\leq \mu\lambda^{1-\kappa}\phi^{1+2\kappa},$$
 where   $0 <\lambda \in L^{1,\infty}(0,T)$ and $\mu> 0$  with
 $$\mu\|\lambda\|_{L^{1,\infty}(0,T)}<\f12.$$
 Then $\phi$ is bounded on $[0,T]$.
 \end{lemma}
  The following  lemma will be frequently used when we  apply Lemma \ref{2.1}.
\begin{lemma}\label{lemma2.2}
Assume that the pair $(p,q)$ satisfies $\f{2}{p}+\f{3}{q}=a$  with
$a,q \geq1 $ and $p>0$. Then, for every $\kappa\in[0,1]$ and given $b,c_0\geq1$, there  exist $p_{\kappa} > 0$ and $\min\{q,b\}\leq q_{\kappa}\leq\max\{q,b\}$ such that
\be\left\{\ba\label{ro}
&\f{2}{p_{\kappa}}+\f{3}{q_{\kappa}}=a, \\
&\f{p_{\kappa}}{q_{\kappa}}=\f{p\big(1-\kappa\big)}{q}+\f{c_0\kappa}{b}.
\ea\right.\ee
\end{lemma}
\begin{proof}
From $\f{2}{p}+\f{3}{q}=a$, we see that
$$\f{p}{q}=\f{1}{3}\big(pa-2\big)$$
Inserting this into $\eqref{ro}_2$, we find that
\be\label{eldeng}
\f{p_{\kappa}}{q_{\kappa}}=\f{1}{3}\big(pa-2\big)\big(1-\kappa\big)+\f{c_0\kappa}{b}.
\ee
This together with $\eqref{ro}_1$ yields that
$p_{\kappa}=p+\kappa(\f{3c_0}{ab}-p +\f2a)$.
Then it follows from \eqref{eldeng} that $q_{\kappa}=\f{ 3pa+3\kappa(\f{3c_0}{b}-pa+2) }{a[pa-2+\kappa(\f{3c_0}{b}-pa+2) ]}.$ The proof of this lemma is   completed.
\end{proof}
\section{Proof of Theorem \ref{the1.1}  and \ref{the1.2}}
\label{sec3}
\setcounter{section}{3}\setcounter{equation}{0}
This section is devoted to proving  Theorem \ref{the1.1} and \ref{the1.2}.
\begin{proof}[Proof of  Theorem \ref{the1.1}]
Multiplying both side of the Navier-Stokes equations \eqref{NS} by $u|u|^{2}$, integrating by parts and  divergence-free
condition, we conclude that
\be\ba\label{2.8}
\f{1}{4}\f{d}{dt}\int_{\mathbb{R}^{3}}|u|^{4}dx+\int_{\mathbb{R}^{3}}|\nabla u|^{2}|u|^{2}dx
+\f{1}{2}\int_{\mathbb{R}^{3}}|\nabla| u|^{2}|^{2}dx&=-\int_{\mathbb{R}^{3}} u\cdot\nabla \Pi |u|^{2}dx =I.
\ea\ee
In what follows, based on \eqref{serrin1}, it suffices to bound $I$ under the hypothesis of Theorem \ref{the1.1}.

(1)
Using integration by parts again, incompressible condition, and Cauchy-Schwarz inequality, we find that
\be\ba
I=&\int_{\mathbb{R}^{3}} \Pi u\cdot\nabla |u|^{2}dx
&\leq C\int_{\mathbb{R}^{3}}\Pi^{2} |u|^{2}dx+\f{1}{8} \int_{\mathbb{R}^{3}} |\nabla u|^{2}|u|^{2}dx.\label{2.10}
\ea\ee
By means of the H\"older inequality \eqref{hiL} or interpolation characteristic \eqref{Interpolation characteristic}
and Sobolev embedding \eqref{sl} in Lorentz spaces,
\be\label{2.11}\Big\| |u|^{2}\Big\|_{L^{\f{2q}{q-1},2}}\leq\| |u|^{2} \Big\|^{1-\f{3}{2q}}_{L^{2}(\mathbb{R}^{3})}\Big\|
|u|^{2}\Big\|^{\f{3}{q}}_{L^{6,2}(\mathbb{R}^{3})} \leq C\Big\| |u|^{2} \Big\|^{1-\f{3}{2q}}_{L^{2}(\mathbb{R}^{3})}\Big\|\nabla
|u|^{2}\Big\|^{\f{3}{2q}}_{L^{2}(\mathbb{R}^{3})}.\ee
With the help of the H\"older inequality \eqref{hiL}, the Calder\'on-Zygmund Theorem and \eqref{2.11},  we infer that
\be\ba\label{2.122}
\int_{\mathbb{R}^{3}}\Pi^{2} |u|^{2}dx &\leq \|\Pi\|_{L^{q,\infty}(\mathbb{R}^{3})}\|\Pi\|_{L^{\f{2q}{q-1},2}(\mathbb{R}^{3})}
\Big\| |u|^{2}\Big\|_{L^{\f{2q}{q-1},2}(\mathbb{R}^{3})}\\
 &\leq C\|\Pi\|_{L^{q,\infty}(\mathbb{R}^{3})}\Big\||u|^{2} \Big\|_{L^{\f{2q}{q-1},2}(\mathbb{R}^{3})}
\Big\| |u|^{2}\Big\|_{L^{\f{2q}{q-1},2}(\mathbb{R}^{3})}\\
&\leq C\|\Pi\|_{L^{q,\infty}(\mathbb{R}^{3})}\Big\| |u|^{2} \Big\|^{2-\f{3}{q}}_{L^{2}(\mathbb{R}^{3})}\Big\|\nabla
|u|^{2}\Big\|^{\f{3}{q}}_{L^{2}(\mathbb{R}^{3})} \\&\leq  C \|\Pi\|^{\f{2q}{2q-3}}_{L^{q,\infty}(\mathbb{R}^{3})}\Big\||u|^{2}\Big\|_{L^{2}(\mathbb{R}^{3})}^{2}
 +\f18\Big\| |u| |\nabla u| \Big\|^{2}_{L^{2}(\mathbb{R}^{3})},
   \ea
\ee
where the Young inequality was used.

Inserting \eqref{2.122} and \eqref{2.10} into \eqref{2.8}, we arrive at
\be\ba\label{2.12}
 \f{d}{dt}\int_{\mathbb{R}^{3}}|u|^{4}dx\leq C \|\Pi\|^{\f{2q}{2q-3}}_{L^{q,\infty}(\mathbb{R}^{3})}\Big\||u|^{2}\Big\|_{L^{2}(\mathbb{R}^{3})}^{2}=C
  \|\Pi\|^{p}_{L^{q,\infty}(\mathbb{R}^{3})}\Big\||u|^{2}\Big\|_{L^{2}(\mathbb{R}^{3})}^{2}.
  \ea
\ee
Thanks to interpolation characteristic \eqref{Interpolation characteristic} or the H\"older inequality \eqref{hiL}, applying lemma  \ref{lemma2.2} with $a=b=2, c_{0}=4$ and \eqref{Lorentzincrease}, we see that
\be\label{2.2000}
 \|\Pi\|^{p_{\kappa}}_{L^{q_{\kappa},\infty}(\mathbb{R}^{3})}\leq
 \|\Pi\|^{p(1-\kappa)}_{L^{q ,\infty}(\mathbb{R}^{3})} \|\Pi\|^{4\kappa}_{L^{2,\infty}}\leq C
 \|\Pi\|^{p(1-\kappa)}_{L^{q ,\infty}(\mathbb{R}^{3})} \|\Pi\|^{4\kappa}_{L^{2 }(\mathbb{R}^{3})}\leq C
 \|\Pi\|^{p(1-\kappa)}_{L^{q ,\infty}(\mathbb{R}^{3})} \Big\||u|^{2}\Big\|^{4\kappa}_{L^{2 }(\mathbb{R}^{3})}.
\ee
Since the pair $(p_{\kappa}, q_{\kappa})$ also meets $2/p_{\kappa}+3/q_{\kappa}=2$, we insert \eqref{2.2000} into \eqref{2.12} to obtain
$$\ba
\f{d}{dt}\|u\|^{4}_{L^{4 }(\mathbb{R}^{3})}\leq C
  \|\Pi\|^{p_{\kappa}}_{L^{q_{\kappa},\infty}(\mathbb{R}^{3})}\Big\||u|^{2}\Big\|_{L^{2}(\mathbb{R}^{3})}^{2}\leq C \|\Pi\|^{p(1-\kappa)}_{L^{q ,\infty}(\mathbb{R}^{3})} \| u \|^{4(1+2\kappa)}_{L^{4 }(\mathbb{R}^{3})}.
\ea$$
Now, wa are in a position to invoke  Lemma \ref{2.1} and \eqref{serrin1} to complete the proof of this part.

(2) From the pressure equations $-\Delta\Pi=\text{div\,div\,}(u\otimes u)$ and the Calder\'on-Zygmund Theorem, we know that
$$\|\nabla\Pi\|_{L^{2}(\mathbb{R}^{3})}\leq C \Big\||u||\nabla u|\Big\|_{L^{2}(\mathbb{R}^{3})}.$$
In view of the interpolation characteristic \eqref{Interpolation characteristic} or the H\"older inequality \eqref{hiL},  one deduces that
\be\label{2.13}
\| u \|_{L^{\f{12q}{3q-2},4}(\mathbb{R}^{3})}\leq \| u \|^{1-\f{1}{q}}_{L^{4 }(\mathbb{R}^{3})}\| u \|^{ \f{1}{q}}_{L^{12,4}(\mathbb{R}^{3})}.
\ee
Sobolev embedding \eqref{sl} in Lorentz spaces leads to
\be\label{2.14}
\|u\|^{2}_{L^{12,4}(\mathbb{R}^{3})}=\Big\||u|^{2}\Big\|_{L^{6,2}(\mathbb{R}^{3})}\leq C\B\||\nabla u| |u| \B\|_{L^{2}(\mathbb{R}^{3})}.
\ee
In the light of the H\"older inequality \eqref{hiL}
\eqref{2.13} and \eqref{2.14}, we have
$$
\ba
I&\leq \B\||\nabla \Pi|^{\f12}\B\|_{L^{4}(\mathbb{R}^{3})}\B\||\nabla \Pi|^{\f12}\B\|_{L^{2q,\infty}(\mathbb{R}^{3})}\B\||u|^{3}\B\|_{L^{\f{4q}{3q-2},\f43}(\mathbb{R}^{3})}\\
&\leq  \| \nabla \Pi  \|^{\f12}_{L^{2}(\mathbb{R}^{3})} \| \nabla \Pi \|^{\f12}_{L^{q,\infty}(\mathbb{R}^{3})} \| u \|^{3}_{L^{\f{12q}{3q-2},4}(\mathbb{R}^{3})}\\
&\leq  C\B\| |u||\nabla u| \B\|^{\f{3+q}{2q}}_{L^{2}(\mathbb{R}^{3})} \| \nabla \Pi \|^{\f12}_{L^{q,\infty}(\mathbb{R}^{3})} \| u \|^{\f{3(q-3)}{q}}_{L^{ 4}(\mathbb{R}^{3}) }.
\ea$$
Combining this with the Young inequality, we see that
$$\ba
I \leq C\| \nabla \Pi \|^{\f{2q}{3q-3}}_{L^{q,\infty}(\mathbb{R}^{3})}\| u \|^{4}_{L^{4 }(\mathbb{R}^{3})}+\f18
\B\||\nabla u| |u| \B\|_{L^{2 }(\mathbb{R}^{3})}^{2}.
\ea$$
Plugging this into \eqref{2.8}, we get
\be\ba \label{2.16} \f{d}{dt}\int_{\mathbb{R}^{3}}|u|^{4}dx +\int_{\mathbb{R}^{3}}|\nabla u|^{2}|u|^{2}dx
 \leq C\| \nabla \Pi \|^{\f{2q}{3q-3}}_{L^{q,\infty}(\mathbb{R}^{3})}\| u \|^{4}_{L^{4 }(\mathbb{R}^{3})}=C\| \nabla \Pi \|^{p}_{L^{q,\infty}(\mathbb{R}^{3})}\| u \|^{4}_{L^{4 }(\mathbb{R}^{3})}.
\ea\ee
In view of interpolation characteristic \eqref{Interpolation characteristic} or the H\"older inequality \eqref{hiL},
\eqref{Lorentzincrease} and Lemma \eqref{lemma2.2}, we infer that
\be\label{2.17}
\| \nabla \Pi \|^{p_{\kappa}}_{L^{q_{\kappa},\infty}(\mathbb{R}^{3})}\leq
\| \nabla \Pi \|^{p(1-\kappa)}_{L^{q,\infty}(\mathbb{R}^{3})}\| \nabla \Pi \|^{c_{1}\kappa}_{L^{2,\infty}(\mathbb{R}^{3})}\leq
\| \nabla \Pi \|^{p(1-\kappa)}_{L^{q,\infty}(\mathbb{R}^{3})}\| \nabla \Pi \|^{c_{1}\kappa}_{L^{2 }(\mathbb{R}^{3})},
\ee
where $c_{1}$ is determined  later.

Notice that $2/p_{\kappa}+3/q_{\kappa}=3$, hence, it follows from \eqref{2.16}, \eqref{2.17} and the Young inequality that
$$\ba
\f{d}{dt}\int_{\mathbb{R}^{3}}|u|^{4}dx +\int_{\mathbb{R}^{3}}|\nabla u|^{2}|u|^{2}dx
 &\leq C\| \nabla \Pi \|^{p_{\kappa}}_{L^{q_{\kappa},\infty}(\mathbb{R}^{3})}\| u \|^{4}_{L^{4 }(\mathbb{R}^{3})}\\
 &\leq C\| \nabla \Pi \|^{p(1-\kappa)}_{L^{q,\infty }(\mathbb{R}^{3})}\| \nabla \Pi \|^{c_{1}\kappa}_{L^{2 }(\mathbb{R}^{3})}\| u \|^{4}_{L^{4 }(\mathbb{R}^{3})}\\
 &\leq C\| \nabla \Pi \|^{p(1-\kappa)}_{L^{q,\infty}(\mathbb{R}^{3})}\Big\| |u||\nabla u| \Big\|^{c_{1}\kappa}_{L^{2 }(\mathbb{R}^{3})}\| u \|^{4}_{L^{4 }(\mathbb{R}^{3})}\\ &\leq C\| \nabla \Pi \|^{\f{2p(1-\kappa)}{2-c_{1}\kappa}}_{L^{q,\infty}(\mathbb{R}^{3})} \| u \|^{\f{8}{2-c_{1}\kappa}}_{L^{4 }(\mathbb{R}^{3})}+\f18\Big\| |u||\nabla u| \Big\|^{2}_{L^{2 }(\mathbb{R}^{3})}. \ea$$
Before going further, we take  $$\delta=\f{(2-c_{1})\kappa}{2-c_{1}\kappa},~~c_{1}=\f{4}{3},$$
in the last relation.
Therefore, we  obtain that
 $$\f{d}{dt}  \| u \|^{4}_{L^{4 }(\mathbb{R}^{3})}\leq C\| \nabla \Pi \|^{p(1-\delta)}_{L^{q,\infty}(\mathbb{R}^{3})} \| u \|^{4(1+2\delta) }_{L^{4 }(\mathbb{R}^{3})}.$$
 Thanks to $\kappa\in[0,1]$, we know that $\delta\in[0,1]$.
Finally,
Lemma \ref{2.1} and  \eqref{serrin1} allow us to finish the proof of Theorem \ref{the1.1}.

\end{proof}

\begin{proof}[Proof of Theorem \ref{the1.2}]

Multiplying the Navier-Stokes system (\ref{NS}) by $\Delta u$, integrating over $\mathbb{R}^{3}$,  div\,$u$=0, and integrating by
parts, we have
\be\ba\label{2.18}
\f12\f{d}{dt}\int_{\mathbb{R}^{3}}|\nabla u|^{2}dx+\int_{\mathbb{R}^{3}}|\nabla^{2}u|dx&\leq \int_{\mathbb{R}^{3}} |\nabla u|^{3}dx.
\ea\ee
Arguing as the same manner in \eqref{2.11}, we find
\be\label{2.19}
\|\nabla u \|^{2}_{L^{\f{2q}{q-1},2}(\mathbb{R}^{3})}\leq C\|  \nabla u \|^{2-\f{3}{q}}_{L^{2}(\mathbb{R}^{3})}\|\nabla^{2} u\|^{\f{3}{q}}_{L^{2}(\mathbb{R}^{3})}.\ee
We derive from the H\"older inequality \eqref{hiL}, \eqref{2.19}  and the Young inequality that
\be\label{2.20}\ba
\int_{\mathbb{R}^{3}} |\nabla u|^{3}dx&\leq \|\nabla u\|_{L^{q,\infty}(\mathbb{R}^{3})}\|\nabla u \|^{2}_{L^{\f{2q}{q-1},2}(\mathbb{R}^{3})} \\
&\leq C\|\nabla u\|_{L^{q,\infty}(\mathbb{R}^{3})}\|  \nabla u \|^{2-\f{3}{q}}_{L^{2}(\mathbb{R}^{3})}\|\nabla^{2} u\|^{\f{3}{q}}_{L^{2}(\mathbb{R}^{3})}\\
&\leq  C \|\nabla u\|^{\f{2q}{2q-3}}_{L^{q,\infty}(\mathbb{R}^{3})}\|\nabla u\|_{L^{2}(\mathbb{R}^{3})}^{2}
 +\f18\| \nabla^{2}u\|^{2}_{L^{2}(\mathbb{R}^{3})}.
\ea\ee
Substituting \eqref{2.20} into  \eqref{2.18}, we show that
\be\label{3.14}
\f{d}{dt}\int_{\mathbb{R}^{3}}|\nabla u|^{2}dx\leq  C \|\nabla u\|^{\f{2q}{2q-3}}_{L^{q,\infty}(\mathbb{R}^{3})}\|\nabla u\|_{L^{2}(\mathbb{R}^{3})}^{2}= C \|\nabla u\|^{p}_{L^{q,\infty}(\mathbb{R}^{3})}\|\nabla u\|_{L^{2}(\mathbb{R}^{3})}^{2}.
\ee
 Along the exact same line as in the proof of \eqref{2.2000}, we find
\be\label{3.15}
 \|\nabla u\|^{p_{\kappa}}_{L^{q_{\kappa},\infty}(\mathbb{R}^{3})}\leq
 \|\nabla u\|^{p(1-\kappa)}_{L^{q ,\infty}(\mathbb{R}^{3})}\| \nabla u \|^{4\kappa}_{L^{2 }(\mathbb{R}^{3})}.
 \ee
  Consequently, we derive from \eqref{3.14} and \eqref{3.15} that
 $$\ba \f{d}{dt}\|\nabla u\|^{2}_{L^2(\mathbb{R}^{3})}
 \leq C \|\nabla u\|^{p_{\kappa}}_{L^{q_{\kappa},\infty}(\mathbb{R}^{3})}\|\nabla u\|_{L^{2}(\mathbb{R}^{3})}^{2}\leq C \|\nabla u\|^{p(1-\kappa)}_{L^{q ,\infty}(\mathbb{R}^{3})} \| \nabla u \|^{2(1+2\kappa)}_{L^{2 }(\mathbb{R}^{3})}.
\ea$$
 Lemma \ref{2.1} and \eqref{serrin2} help  us  achieve the proof of  Theorem \ref{the1.2}.
 \end{proof}

\section*{Acknowledgement}
Wang was partially supported by  the National Natural
Science Foundation of China under grant (No. 11971446  and No. 11601492)  and the Youth Core Teachers Foundation of  Zhengzhou University of Light Industry.
Wei was partially supported by the National Natural Science Foundation of China under grant ( No. 11601423, No. 11701450, No. 11701451, No. 11771352, No. 11871057) and Scientific Research Program Funded by Shaanxi Provincial Education Department (Program No. 18JK0763).

\end{document}